\documentclass[12pt]{article}

\usepackage[utf8]{inputenc}
\usepackage[T1]{fontenc}
\usepackage[margin=1in]{geometry}
\usepackage{amsmath, amsfonts, amssymb, amsthm}
\usepackage[none]{hyphenat}
\usepackage{graphicx}
\usepackage{float}
\usepackage{times}
\usepackage{lipsum}

\newcommand{\R}{\mathbb R}

\newtheorem{Th}{Theorem}[section]
\newtheorem{df}{Definition}[section]
\newtheorem{pro}{Proposition}[section]
\newtheorem{lm}{Lemma}[section]

\newtheorem{cor}{Corollary}[section]

\newtheorem{rma}{Remark}[section]

\newcommand{\norme}[1]{\left\Vert #1\right\Vert}

\DeclareMathOperator*{\essinf}{ess\,inf}

\def\Xint#1{\mathchoice
{\XXint\displaystyle\textstyle{#1}}%
{\XXint\textstyle\scriptstyle{#1}}%
{\XXint\scriptstyle\scriptscriptstyle{#1}}%
{\XXint\scriptscriptstyle\scriptscriptstyle{#1}}%
\!\int}
\def\XXint#1#2#3{{\setbox0=\hbox{$#1{#2#3}{\int}$ }
\vcenter{\hbox{$#2#3$ }}\kern-.6\wd0}}

\newcommand\restr[2]{{
  \left.\kern-\nulldelimiterspace 
  #1 
  \vphantom{\big|} 
  \right|_{#2} 
  }}
  
\newcommand\lfrac[2]{\frac{\displaystyle #1}{\displaystyle #2}}

\numberwithin{equation}{section}

\begin{document}
\title{\bf{PWB-method and Wiener criterion for boundary regularity under generalized Orlicz growth}}
\author{\small{By Allami BENYAICHE and Ismail KHLIFI.}\\
\small {\it allami.benyaiche@uit.ac.ma; is.khlifi@gmail.com}\\
\small{Ibn Tofail University, Department of Mathematics, B.P: 133, Kenitra-Morocco.}}
\date{}
\maketitle {\bf Abstract:} {\it Perron's method and Wiener's criterion have entirely solved the Dirichlet problem for the Laplace equation. Since then, this approach has attracted the attention of many mathematicians for applying these ideas in the more general equations. So, in this paper, we extend the Perron method and the Wiener criterion to the $G(\cdot)$-Laplace equation.}

\section{Introduction}
In this paper, we are concerned with the regularity of boundary point of a bounded domain  $\Omega $ of $\R^n$ respect to Dirichlet problem associated to $G(\cdot)$-Laplace operator defined by: 
$$-\Delta_{G(\cdot)}(u) := - \text{div} \frac{g(x,|\nabla u|)}{|\nabla u|}\nabla u,$$
where $g(\cdot)$ is the density of a generalized Orlicz function $G(\cdot)$ that have been previously used in $\cite{ref3,ref4,ref5,ref12,ref13,ref21}$. This equation covers for example, the p-Laplace equation $G(x,t) = t^p$, the variable exponent case $G(x,t) = t^{p(x)}$ and its perturbation $G(x,t) = t^{p(x)}\log(e+t)$, the double phase case $G(x,t) = t^p + a(x)t^q$, and the Orlicz case $G(x,t) = G(t)$. More examples can be found in $\cite{ref12}$.\\
\\
Historically, Riemann proposed in $1851$ the Dirichlet principle, which states that a harmonic function always exists in the interior of a domain with boundary conditions given by a continuous function. However, Lebesgue produced in $1912$ an example of the bounded domain on which the Dirichlet problem was not always solvable. Overcome this problem; there is a method based on the work of Perron, Wiener, and  Brelot is nowadays well known the Perron's method or PWB-method $\cite{ref23}$, also referred to the method of subharmonic functions,  based on the finding the largest subharmonic function with boundary values below the desired values. The advantage of this method is that one can construct reasonable solutions for arbitrary boundary data. After that, in $1924$, Wiener introduced the harmonic capacity to give his famous criterion of the regularity of a boundary point which allows us to solve the Dirichlet problem for the Laplace equation completely. Since then, Perron's method and Wiener's criterion have attracted the attention of many mathematicians for applying these ideas to study the Dirichlet problem in the more general equations.\\

For $f \in W^{1,G(\cdot)}(\Omega)$, the authors proved in $\cite{ref5}$ the existence of the solution to the Dirichlet-Sobolev problem 
$$\left \{
   \begin{array}{r c l}
    -\Delta_{G(\cdot)}(u) = 0 \quad \text{in} \quad  \Omega\\
    u - f \in W_0^{1,G(\cdot)}(\Omega).
   \end{array}
   \right.$$
where $W^{1,G(\cdot)}(\Omega)$ and $W_0^{1,G(\cdot)}(\Omega)$ are the generalized Orlicz-Sobolev space, also called Musielak-Orlicz-Sobolev space (see section $2$). So, the question that arises is on the regularity of the Sobolev boundary point $x_0 \in \partial \Omega$, i.e 
$$
\lim_{x \to x_0}u(x) = f(x_0),
$$
for any $f \in W^{1,G(\cdot)}(\Omega) \cap C(\overline{\Omega})$.\\
In the p-Laplace equation, $G(x,t) = t^p$, if $\Omega$ satisfies the exterior sphere condition (see section $3$) then $\Omega$ is a Sobolev $p$-regular domain. By the work of Harjulehto and H\"ast\"o in the locally fat set $\cite{ref13}$, we generalize this result in our situation. As a consequence of this result, we solve the Dirichlet problem for simple domains. We shall need this possibility to construct the Poison modification of our functions because this modification is based on the approximation of the solution to the Dirichlet problem in balls. Therefore, by the ideas of Granlund, Lindqvist, and Martio $\cite{ref11}$, we can apply Perron's method to the $G(\cdot)$-Laplace equation. More correctly, the regularity of boundary point is defined in connection with the solution of generalized Dirichlet problem (see $\cite{ref26,ref23}$), not only for Dirichlet-Sobolev solution. Precisely, we say a boundary point $x_0 \in \partial \Omega$ is regular if 
$$\lim_{x \to x_0}H_f(x) = f(x_0),$$ 
for $f \in C(\partial \Omega)$ where $H_f$ is the Perron solution with boundary data $f$ (see section $5$).\\
In the non-linear case, the best condition for the regularity boundary points is given by the celebrated Wiener criterion.  This criterion has been generalized in the variable constant. The sufficiency part has been proved by Maz'ya in $\cite{ref18}$, and the necessary part was proved by Kilpelainen and Maly in $\cite{ref15}$. Next, Trudinger and Wang $\cite{ref20}$ gave a new method based on Poisson modification and Harnack inequality. Mikkonen has treated the weighted situation in $\cite{ref24}$. Björn has developed the proof of this criterion in the metric measure spaces $\cite{ref7}$. In the variable exponent case, $G(x,t) = t^{p(x)}$, the problem has been study by Alkhutov and Krasheninnikova in $\cite{ref1}$. Recently, K.A Lee and  S.C Lee in $\cite{ref25}$ proved the Wiener criterion for the regularity of Sobolev boundary point in the Orlicz case. So, it is natural to ask what Wiener criterion should satisfy to guarantee regular points in the generalized Orlicz situation. Applying estimates of a particular $G(\cdot)$-supersolution called the $G(\cdot)$-potential, the central condition $(A_{1,n})$ (see section $2$), the pointwise Wolff estimates in $\cite{ref6}$, and the Perron $G(\cdot)$-solution, we get our main result, which is new even in the Orlicz case. 
\begin{Th}
 Let $ G(\cdot) \in \Phi(\R^n) \cap C^1(\R^+)$ be strictly convex and satisfy $(SC)$, $(A_0)$, $(A_1)$, and $(A_{1,n})$. The point $x_0 \in \partial \Omega$ is ${G(\cdot)}$-regular if and only if for some $\rho > 0$, 
$$\int_0^{\rho} g^{-1}\left(x_0 , \frac{\textsl{cap}_{G(\cdot)}(B(x_0,t)\cap \Omega^\complement, B(x_0,2t)}{t^{n-1}}\right) \, \mathrm{d}t = \infty.$$
\end{Th}

\section{Preliminary}

\begin{df}
A function $G: \Omega \times [0,\infty) \rightarrow [0,\infty]$ is called a generalized $\Phi$-function, denoted by $G(\cdot) \in \Phi(\Omega)$, if the following conditions hold
\begin{itemize}
\item For each $t \in [0,\infty)$, the function $G(\cdot,t)$ is measurable.
\item For a.e $x \in \Omega$, the function $G(x, \cdot)$ is an $\Phi$-function, i.e.
\begin{enumerate}
\item $G(x,0) =  \displaystyle \lim\limits_{t \rightarrow 0^+} G(x,t) = 0$ and $\displaystyle \lim\limits_{t \rightarrow \infty} G(x,t) = \infty$;
\item $G(x,\cdot)$ is increasing and convex.
\end{enumerate}
\end{itemize}
\end{df}

\noindent Note that, a generalized $\Phi$-function can be represented as
$$ G(x,t) =  \displaystyle \int_{0}^t g(x,s) \, \mathrm{d}s,$$
where $g(x,\cdot)$ is the right-hand derivative of $G(x,\cdot)$. Furthermore, for each $x \in \Omega$, the function $g(x,\cdot)$ is right-continuous and nondecreasing. So, we have the following inequality
\begin{equation}
g(x,a)b \leq g(x,a)a + g(x,b)b, \; \; \text{for} \, x \in \Omega \; \text{and} \; a,b \geq 0
\end{equation}
We denote $G^+_{B}(t) := \sup_B G(x,t), \; G^-_B(t) := \inf_B G(x,t)$.  We say that  $G(\cdot)$ satisfies\\
$ (SC):$ If there exist two constants  $g_0,g^0 > 1$ such that,
$$ 1 < g_0 \leq \displaystyle \frac{t g(x,t)}{G(x,t)} \leq g^0.$$
$(A_0):$ If there exists a constant  $c_0 > 1$ such that,
 $$ \displaystyle \frac{1}{c_0} \leq G(x,1) \leq c_0, \, \; \text{a.e} \; \, x \in \Omega.$$
 $(A_{1}):$ If there exists $C > 0$ such that, for every $ x,y \in B_R \subset \Omega$ with $R \leq 1$, we have
$$ G_B(x,t) \leq C G_B(y,t), \; \; \; \text{when} \; \; G^-_B(t) \in \left[ 1 , \frac{1}{R^n}\right].$$
$(A_{1,n}):$ If there exists $C > 0$ such that, for every $ x,y \in B_R \subset \Omega$ with $R \leq 1$, we have
$$ G_B(x,t) \leq C G_B(y,t), \; \; \; \text{when} \; \; t \in \left[ 1 , \frac{1}{R}\right].$$

The following lemma gives a more flexible characterization of $(A_{1,n}) \; \cite{ref12}$.

\begin{lm}
Let $\Omega \subset \R^n$ be convex, $G(\cdot) \in \Phi(\Omega)$ and $0 < r \leq s$. Then $G(\cdot)$ satisfies $(A_{1,n})$ if, and only if, there exists $C > 0$ such that, for every $ x,y \in B_R \subset \Omega$ with $R \leq 1$, we have
$$ G_B(x,t) \leq CG_B(y,t) \; \; \; \text{when} \; \; t \in \left[ r , \frac{s}{R}\right]. $$
\end{lm}

\noindent  Under the structure condition $(SC)$, we have the following inequalities
\begin{equation}
\sigma^{g_0} G(x,t) \leq G(x,\sigma t) \leq \sigma^{g^0} G(x,t), \; \;\text{for} \; x \in \Omega, \; \, t \geq 0 \; \text{and} \;\sigma \geq 1.
\end{equation}
\begin{equation}
\sigma^{g^0} G(x,t) \leq G(x,\sigma t) \leq \sigma^{g_0} G(x,t), \; \;\text{for} \; x \in \Omega, \; \, t \geq 0 \; \text{and} \;\sigma \leq 1.
\end{equation}
We define $G^*(\cdot)$ the conjugate $\Phi$-function of $G(\cdot)$, by
$$ G^*(x,s) := \sup_{t \geq 0}(st - G(x,t)), \; \, \; \text{for} \; x \in \Omega \; \text{and} \; s \geq 0.$$ 
Note that $G^*(\cdot)$ is also a generalized $\Phi$-function and can be represented as
$$ G^*(x,t) =  \displaystyle \int_{0}^t g^{-1}(x,s) \, \mathrm{d}s,$$
with $g^{-1}(x,s) : = \sup \{ t \geq 0 \; : \; g(x,t) \leq s \}$. Furthermore, if $G(\cdot)$ satisfies $(SC)$, then $G^*(\cdot)$ satisfies also $(SC)$, as follows
\begin{equation}
\displaystyle \frac{g^0}{g^0 - 1} \leq \displaystyle \frac{t g^{-1}(x,t)}{G^*(x,t)} \leq \frac{g_0}{g_0 - 1}.
\end{equation}
The functions $G(\cdot)$ and $G^*(\cdot)$ satisfies the following Young inequality 
$$ st \leq G(x,t) + G^*(x,s), \; \, \text{for} \; x \in \Omega \; \text{and} \; s,t \geq 0.$$
Further, we have the equality if $s = g(x,t)$ or  $t = g^{-1}(x,s)$. So, if $G(\cdot)$ satisfies $(SC)$, we have the following inequality
\begin{equation}
 G^*(x,g(x,t)) \leq (g^0-1)G(x,t), \; \forall x \in \Omega, t \geq 0.
\end{equation}

\begin{df}
We define the generalized Orlicz space, also called Musielak-Orlicz space, by
$$ L^{G(\cdot)}(\Omega) := \{ u \in L^0(\Omega) \; : \; \displaystyle \lim\limits_{\lambda \rightarrow 0}  \rho_{G(\cdot)}(\lambda|u|) = 0 \},$$
where $\rho_{G(\cdot)}(t) = \displaystyle \int_{\Omega} G(x,t) \, \mathrm{d}x$. If $G(\cdot)$ satisfies $(SC)$, then
$$ L^{G(\cdot)}(\Omega) = \{ u \in L^0(\Omega) \; : \; \rho_{G(\cdot)}(|u|) < \infty \}.$$
\end{df}

\begin{df}
We define the generalized Orlicz-Sobolev space by
$$ W^{1,G(\cdot)}(\Omega) := \{ u \in L^{G(\cdot)}(\Omega) \; : \; |\nabla u| \in L^{G(\cdot)}(\Omega), \; \, \text{in the distribution sense} \},$$
equipped with the norm
$$ \norme{u}_{1,G(\cdot)} = \norme{u}_{G(\cdot)} + \norme{\nabla u}_{G(\cdot)}.$$
\end{df}

\begin{df}
$W^{1,G(\cdot)}_0(\Omega)$ is the closure of $C^\infty_0(\Omega)$ in $W^{1,G(\cdot)}(\Omega).$
\end{df}

Note that, in such spaces, we have the following Poincaré inequality $\cite{ref12}$.

\begin{Th}
Let $G(\cdot) \in \Phi(\Omega) $ satisfy $(A_0)$ and $(A_1)$. There exists a constant $C>0$ such that
$$\int_{\Omega} G(x, \frac{|u|}{\text{diam}(\Omega)}) \, \mathrm{d}x  \leq C \left( \int_{\Omega} G(x, |\nabla u|) \, \mathrm{d}x  + |\{\nabla u\ \neq 0\} \cap \Omega|\right),$$
for every $u \in W^{1,G(\cdot)}_0(\Omega)$. 
\end{Th}

\begin{df}
Let $G(\cdot) \in \Phi(\Omega)$ and $K \subset \Omega$ be a compact set. The relative $G(\cdot)$-capacity of $K$ with respect to $\Omega$ is 
$$\textsl{cap}_{G(\cdot)}(K;\Omega) = \inf_{u \in S_{G(\cdot)}(K;\Omega)} \displaystyle \int_{\Omega} G(x, |\nabla u|) \, \mathrm{d}x $$
where $S_{G(\cdot)}(K;\Omega)=\{u\in W^{1, G(\cdot)}_0(\Omega) \; : \; u \geq 1 \; \text{on} \; K\}$
\end{df}

\begin{pro}
Let $G(\cdot) \in \Phi(\Omega)$ satisfy $(A_0)$ and $(A_{1})$..
\begin{itemize}
\item[i)] $\textsl{cap}_{G(\cdot)}(\emptyset;\Omega)=0$.
\item[ii)] If $K , K^\prime$ are compact sets and  $\Omega^\prime$ is open set such that $K \subset K^\prime \subset \Omega^\prime \subset \Omega$, then 
$$\textsl{cap}_{G(\cdot)}(K;\Omega) \leq \textsl{cap}_{G(\cdot)}(K^\prime;\Omega^\prime).$$
\item[iii)] If $K \subset  B(x_0,r)$ and $0 < r \leq s \leq 2r$, then
$$\textsl{cap}_{G(\cdot)}(K; B(x_0,2s)) \leq \textsl{cap}_{G(\cdot)}(K; B(x_0,2r)) \leq C \left(\textsl{cap}_{G(\cdot)}(K; B(x_0,2s)) + s^{n}\right).$$
\end{itemize}
\end{pro}

\begin{proof}
For i) and ii), we can see $\cite{ref2}$.\\
iii) Since the first inequality is trivial, it suffices to verify the second inequality in the extremal case $s = 2r$. Let $\eta \in C^\infty_c(B(x,2r))$ such that $0\leq \eta \leq 1$, $\eta =1$ in $B(x,r)$ and $|\nabla \eta| \leq \displaystyle \frac{C}{r}$. If $u \in S_{G(\cdot)}(K;B(x,4r))$, then $\eta u \in S_{G(\cdot)}(K;B(x,2r))$, so by Theorem $2.1$, we have
$$\begin{array}{ll}
\textsl{cap}_{G(\cdot)}(K; B(x_0,2r)) & \leq \displaystyle \int_{B(x_0,2r)} G(x, |\nabla \eta u|) \, \mathrm{d}x\\
& \leq C\left(\displaystyle \int_{B(x_0,2r)} G(x, \eta |\nabla u|) \, \mathrm{d}x + \displaystyle \int_{B(x_0,2r)} G(x, u |\nabla \eta |) \, \mathrm{d}x\right)\\
& \leq C\left(\displaystyle \int_{B(x_0,4r)} G(x, |\nabla u|) \, \mathrm{d}x + \displaystyle \int_{B(x_0,4r)} G(x, \frac{u}{r}) \, \mathrm{d}x\right)\\
& \leq C\left(\displaystyle \int_{B(x_0,4r)} G(x, |\nabla u|) \, \mathrm{d}x + r^n\right)\\
\end{array}$$
Taking the infimum over all such functions $u$, we obtain
$$\textsl{cap}_{G(\cdot)}(K; B(x_0,2r)) \leq C \left(\textsl{cap}_{G(\cdot)}(K; B(x_0,4r)) + r^{n}\right).$$
This concludes the proof.
\end{proof}

\section{ $G(\cdot)$-Laplace equation}

Let $G(\cdot) \in \Phi(\Omega)$, we consider the following $G(\cdot)$-Laplace equation
\begin{equation}
 - \text{div} \frac{g(x,|\nabla u|)}{|\nabla u|}\nabla u = 0.
\end{equation}

\begin{df}
A function $h \in W^{1,G(\cdot)}(\Omega)$ is $G(\cdot)$-harmonic in $\Omega$ if it is continuous and $G(\cdot)$-solution to equation $(3.1)$ in $\Omega$ i.e
$$ \int_{\Omega} \frac{g(x,|\nabla h|)}{|\nabla h|} \nabla h \cdot \nabla \varphi \, \mathrm{d}x = 0,$$
whenever  $\varphi \in  W^{1,G(\cdot)}_0(\Omega)$.
\end{df}

\begin{df}
A function $u \in W^{1,G(\cdot)}(\Omega)$ is a $G(\cdot)$-supersolution (resp, $G(\cdot)$-subsolution) to equation $(3.1)$ in $\Omega$ if
$$ \int_{\Omega} \frac{g(x,|\nabla u|)}{|\nabla u|} \nabla u \cdot \nabla \varphi \, \mathrm{d}x \geq 0 \; \; (\text{resp,} \leq 0),$$
whenever  $\varphi \in W^{1,G(\cdot)}_0(\Omega)$ and nonnegative.
\end{df}

Given $v_0 \in W^{1,G(\cdot)}(\Omega)$ and $\psi$: $\Omega \to [-\infty, \infty]$ be any function. Construct the obstacle set:
$$ \mathcal{K}_{\psi , v_0}(\Omega) = \{ u \in W^{1,G(\cdot)}(\Omega) \ : \; u \leq \psi, \; \; \text{a.e in} \; \Omega \; \; \text{and} \; \; u - v_0 \in W_0^{1,G(.)}(\Omega) \}.$$
By Theorem $4.1$ in $\cite{ref5}$, if  $\mathcal{K}_{\psi , v_0}(\Omega)$ is not empty then there exits $ u \in \mathcal{K}_{\psi , v_0}(\Omega)$ such that
$$  \displaystyle \int_{\Omega} \frac{g(x,|\nabla u|)}{|\nabla u|} \nabla u \cdot (\nabla u - \nabla v) \, \mathrm{d}x \geq 0,$$
whenever $v \in \mathcal{K}_{\psi , v_0}(\Omega)$. We said $u$ is a solution of the obstacle problem.

\begin{Th}
Let $ G(\cdot) \in \Phi(\Omega)$ satisfies $(SC)$. Then for every $v_0 \in W^{1,G(\cdot)}(\Omega)$, there exists $u \in W^{1,G(\cdot)}(\Omega)$ a $G(\cdot)$-solution to equation $(3.1)$ in $\Omega$, such that $u - v_0 \in W^{1,G(\cdot)}_0(\Omega)$.\\
If $G(\cdot)$ is strictly convex and satisfies $(A_0)$, the $G(\cdot)$-solution is unique, and if $(A_1)$, $(A_{1,n})$ hold, then it is continuous.
\end{Th}

\begin{proof}
Let $ G(\cdot) \in \Phi(\Omega)$ satisfies $(SC)$ and $v_0 \in W^{1,G(\cdot)}(\Omega)$ so $\mathcal{K}_{\infty , v_0}(\Omega) \neq \emptyset$. Then there exists a solution $u$ of the obstacle problem in $\mathcal{K}_{\infty , v_0}(\Omega)$.\\
Let  $\varphi \in W_0^{1,G(\cdot)}(\Omega)$ then  $u - \varphi, u + \varphi \in \mathcal{K}_{\infty , v_0}(\Omega)$.\\
Hence
$$ \int_{\Omega} \lfrac{g(x,|\nabla u|)}{|\nabla u|} \nabla u \cdot \nabla \varphi \, \mathrm{d}x \geq 0$$
and
$$ -\int_{\Omega} \lfrac{g(x,|\nabla u|)}{|\nabla u|} \nabla u \cdot \nabla \varphi  \, \mathrm{d}x \geq 0.$$
Consequently
$$ \int_{\Omega} \lfrac{g(x,|\nabla u|)}{|\nabla u|} \nabla u \cdot \nabla \varphi  \, \mathrm{d}x = 0$$
whenever $\varphi \in W_0^{1,G(\cdot)}(\Omega).$ Then $u$ is a $G(\cdot)$-solution to equation $(3.1)$ in $\Omega$ such that $u - v_0 \in W^{1,G(\cdot)}_0(\Omega)$.\\ 
When $ G(\cdot)$ is strictly convex and satisfies $(A_0)$, using the comparison weak principle Lemma $4.3$ in $\cite{ref5}$ the $G(\cdot)$-solution is unique. If $G(\cdot)$ satisfy $(A_1)$, $(A_{1,n})$ by Corollary $4.1$ in $\cite{ref4}$ a locally bounded $G(\cdot)$-solution is locally Hölder continuous. This concludes the proof.
\end{proof}

\section{Sobolev $G(\cdot)$-regular boundary points and exterior sphere condition}

\begin{df}
Let $G(\cdot) \in \Phi(\Omega)$  strictly convex and satisfy $(SC)$, $(A_0)$,  $(A_1)$ and $(A_{1,n})$. A boundary point $x_0$ of a bounded open set $\Omega$ is said to be Sobolev $G(\cdot)$-regular if, for each function $v_0 \in W^{1,G(\cdot)}(\Omega) \cap C(\overline{\Omega})$, the $G(\cdot)$-harmonic function $h$ in $\Omega$ with $h - v_0 \in W_0^{1,G(\cdot)}(\Omega)$ satisfies
$$\lim_{x \to x_0}h(x) = h(x_0).$$
Furthermore, we say that a bounded open set $\Omega$ is Sobolev $G(\cdot)$-regular if each $x_0 \in \partial \Omega$.
\end{df}

In $\cite{ref13}$, Harjulehto and H\"ast\"o gave the following sufficient condition for the Sobolev $G(\cdot)$-regular point.

\begin{Th}
Let $x_0 \in \partial \Omega$. Let $ G(\cdot) \in \Phi(\R^n)$ be strictly convex and satisfy $(SC)$, $(A_0)$, $(A_1)$, and $(A_{1,n})$. If there exists $C \in (0,1)$ and $R > 0$ such that
$$\textsl{cap}_{G(\cdot)}(B(x_0,r) \backslash \Omega ; B(x_0,2r))\geq C \textsl{cap}_{G(\cdot)}(B(x_0,r) ; B(x_0,2r))  \quad \text{for all} \quad 0 <r<R.$$
Then $x_0$ is a Sobolev $G(\cdot)$-regular point.
\end{Th}

\begin{df}
We say that a boundary point $x_0$ of a bounded open set $\Omega$ satisfies the exterior sphere condition, if there is a ball $B(y_0, \rho)$ such that $B(y_0, \rho) \cap \overline{\Omega} = \{x_0\}$.\\
Furthermore, we say that a bounded open set $\Omega$  satisfies the exterior sphere condition if each $x_0 \in \partial \Omega$.
\end{df}

\begin{lm}
Let $ G(\cdot) \in \Phi(\sigma B)$ with $\sigma > 1$ satisfies $(SC)$. Then there exits a positive constant $C = C(n,g^0,g_0,\sigma)$ such that
$$\frac{1}{C}|B|G^-_{\sigma B}\left(\frac{1}{r}\right) \leq \textsl{cap}_{G(\cdot)}(B; \sigma B) \leq C|B|G^+_{\sigma B}\left(\frac{1}{r}\right).$$
\end{lm}

\begin{proof}
Let  $u \in W_0^{1,G(\cdot)}(\sigma B)$ be such that $0 \leq u \leq 1$, $u=1$ in $B$ and $|\nabla u | \leq \lfrac{C}{r}$. Then by the condition $(SC)$, we have
$$\textsl{cap}_{G(\cdot)}(B; \sigma B) \leq \displaystyle \int_{\sigma B} G(x, |\nabla u|) \, \mathrm{d}x \leq \displaystyle \int_{\sigma B} G_{\sigma B}^+\left(\lfrac{C}{r}\right) \, \mathrm{d}x \leq C |B| G_{\sigma B}^+\left(\lfrac{1}{r}\right).$$
For the opposite inequality by Jensen-type inequality in $\cite{ref12}$ and the definition of $1$-capacity that 
$$\begin{array}{ll}
\displaystyle \int_{\sigma B} G(x, |\nabla u|) \, \mathrm{d}x & \geq \displaystyle \int_{\sigma B} G_{\sigma B}^-(|\nabla u|) \, \mathrm{d}x\\
& = |\sigma B| \displaystyle \Xint-_{\sigma B} G_{\sigma B}^-(|\nabla u|) \, \mathrm{d}x\\
& \geq |\sigma B| G_{\sigma B}^-\left( \displaystyle \Xint-_{\sigma B} |\nabla u| \, \mathrm{d}x\right)\\
& \geq C |B| G_{\sigma B}^-\left(\lfrac{\textsl{cap}_{1}(B; \sigma B)}{|\sigma B|}\right)
\end{array}$$ 
Since by Example $ 2.12$ in $\cite{ref14}$ we have $\textsl{cap}_{1}(B; \sigma B) = Cr^{n-1}$, 
then by the condition $(SC)$, we get
$$\displaystyle \int_{\sigma B} G(x, |\nabla u|) \, \mathrm{d}x  \geq C |B| G_{\sigma B}^-\left( \lfrac{1}{r} \right).$$ 
This concludes the  proof. 
\end{proof}

\begin{Th}
Let $G(\cdot) \in \Phi(\Omega)$  strictly convex and satisfy $(SC)$, $(A_0)$,  $(A_1)$ and $(A_{1,n})$. If $\Omega$ satisfies the exterior sphere condition, then $\Omega$ is Sobolev $G(\cdot)$-regular.
\end{Th}

\begin{proof}
Let $\Omega$ satisfies the exterior sphere condition. Then for every $x_0 \in \partial \Omega$ there exists a ball $B(y_0, r)$ such that $B(y_0, r) \cap \overline{\Omega} = \{x_0\}$. So we have $B(x_0,3r) \backslash \Omega$ contains $B(y_0, r)$. Then, by Proposition $2.1$ and Lemma $4.1$, we have
$$\begin{array}{ll}
\textsl{cap}_{G(\cdot)}(B(x_0,3r) \backslash \Omega ; B(x_0,6 r)) & \geq C \textsl{cap}_{G(\cdot)}(B(y_0,r) ; B(x_0,6r))\\
& \geq C \textsl{cap}_{G(\cdot)}(B(y_0,r) ; B(y_0,8r))\\
& \geq C |B(y_0,r)|G^-_{B(y_0,8r)}\left(\displaystyle \frac{1}{r}\right)\\
& \geq C |B(x_0,3r)|G^-_{B(y_0,8r)}\left(\displaystyle \frac{1}{r}\right)
\end{array}$$ 
By the condition $(A_{1,n})$ there exists a constant $C>0$ such that
$$G^+_{B(y_0,8r)}\left(\frac{1}{r}\right) \leq C G^-_{B(y_0,8r)}\left(\frac{1}{r}\right)$$
Using again Lemma $4.1$ we obtain
$$\begin{array}{ll}
\textsl{cap}_{G(\cdot)}(B(x_0,3r) \backslash \Omega ; B(x_0,6r)) & \geq C |B(x_0,3r)|G^+_{B(y_0,8r)}\left(\displaystyle \frac{1}{r}\right)\\
& \geq  C |B(x_0,3r)|G^+_{B(x_0,6r)}\left(\displaystyle \frac{1}{r}\right)\\
& \geq C \textsl{cap}_{G(\cdot)}(B(x_0,3r) ; B(x_0,6r))
\end{array}$$ 
for $r$ small enough, so by Theorem $4.1$ we have $\Omega$ is Sobolev $G(\cdot)$-regular.
\end{proof}

\begin{cor}
Let $G(\cdot) \in \Phi(\Omega)$  strictly convex and satisfy $(SC)$, $(A_0)$,  $(A_1)$ and $(A_{1,n})$. All balls are Sobolev $G(\cdot)$-regular. 
\end{cor}

Consequently, every open set can be exhausted by Sobolev $G(\cdot)$-regular open sets as a consequence of this corollary.

\section{The Perron-Wiener-Brelot method}

\subsection{Upper and lower Perron $G(\cdot)$-solution}

Let $G(\cdot) \in \Phi(\Omega)$. A function $u : \Omega \rightarrow \R \cup \{\infty\}$ is called $G(\cdot)$-superharmonic in $\Omega$ if
\begin{enumerate}
\item[i)] $u$ is lower semicontinuous,
\item[ii)] $u \not\equiv \infty$ in $\Omega$,
\item[iii)] for each domain $D \subset \subset \Omega$ the comparison principle holds: if $h \in C(\overline{D})$ is $G(\cdot)$-harmonic in $D$ and $u \geq h$ on $\partial D$ then $u \geq h$ in $D$.
\end{enumerate}
A function $v : \Omega \rightarrow \R \cup \{-\infty\}$ is called $G(\cdot)$-subharmonic in $\Omega$ if 
\begin{enumerate}
\item[i)] $u$ is upper semicontinuous,
\item[ii)] $u \not\equiv -\infty$ in $\Omega$,
\item[iii)] for each domain $D \subset \subset \Omega$ the comparison principle holds: if $h \in C(\overline{D})$ is $G(\cdot)$-harmonic in $D$ and $u \leq h$ on $\partial D$ then $u \leq h$ in $D$.
\end{enumerate}

For $f : \partial \Omega \rightarrow [-\infty,\infty]$ be a function, we define as in classical potential theory $\cite{ref26}$ two classes of functions:
\begin{enumerate}
\item[•] The upper class $U_f$ consists of all functions $v: \Omega \to \left(-\infty, \infty\right]$ such that
\begin{enumerate}
\item[i)] $v$ is $G(\cdot)$-superharmonic in $\Omega$,
\item[ii)] $v$ is bounded below, 
\item[iii)] $\liminf_{x\to \xi}v(x) \geq f(\xi)$ when $\xi \in \partial \Omega$.
\end{enumerate}

\item[•] The lower class $L_f$ consists of all functions $u: \Omega \to \left[-\infty, \infty\right)$ such that
\begin{enumerate}
\item[i)] $u$ is $G(\cdot)$-subharmonic in $\Omega$,
\item[ii)] $u$ is bounded above,
\item[iii)] $\limsup_{x\to \xi}u(x) \leq f(\xi)$ when $\xi \in \partial \Omega$.
\end{enumerate}
\end{enumerate}

We define at each point in $\Omega$
$$\text{The upper Perron} \; G(\cdot)\text{-solution} \; \;  \overline{H}_f(x) = \inf_{v \in U_f}v(x)$$
$$\text{The lower Perron } \; G(\cdot)\text{-solution} \; \;  \underline{H}_f(x) = \sup_{v \in L_f}v(x)$$
If $U_f = \emptyset$ (or $L_f=\emptyset$), then we have $\overline{H}_f = \infty$ (and $\underline{H}_f=-\infty$ respectively).\\

The following lemma gives simple properties for Perron $G(\cdot)$-solutions.

\begin{lm}
Let $f : \partial \Omega \rightarrow [-\infty,\infty]$ be a function, we have the following properties
\begin{enumerate}
\item[1)] $\underline{H}_f = -\overline{H}_{-f}$
\item[2)] $\underline{H}_f \leq \overline{H}_{f}$
\item[3)] if $f \leq g$, then $\overline{H}_{f} \leq \overline{H}_{g}$ 
\item[4)] for $\lambda \in R$, we have $\overline{H}_{f + \lambda} = \overline{H}_{f} + \lambda$ and $\overline{H}_{\lambda f} = \lambda \overline{H}_{f}$
\end{enumerate}
For $3)$ and $4)$, a similar statement is true if $\overline{H}_f$ is replace by $\underline{H}_f$.
\end{lm}

\subsection{The Poisson modification}

Generally, to construct the Poisson modification, the Harnack Convergence theorem and the comparison principle are needed (see $\cite{ref22}$).

\begin{Th}[Harnack Convergence theorem]
Let $G(\cdot) \in \Phi(\Omega)$ satisfies $(SC)$. Suppose that $u_i$ is a $G(\cdot)$-harmonic such that 
$$ 0 \leq u_1 \leq u_2 \leq ... , \; u = \lim u_i, \; \text{pointwise in} \; \Omega.$$
Then, either $u = \infty$ or $u$ is a $G(\cdot)$-harmonic in $\Omega$.
\end{Th}

\begin{lm}[Comparison principle]
Let $ G(\cdot) \in \Phi(\Omega)$ satisfies $(SC)$. Suppose that $u$ is a $G(\cdot)$-subharmonic and $v$ is a $G(\cdot)$-superharmonic in $\Omega$ such that
$$\limsup_{x \to y}u(x) \leq \liminf_{x \to y}v(x)$$
for all $y \in \partial \Omega$. If the left and right-hand sides are neither $\infty$ nor $-\infty$ at the same time, then 
$$u \leq v \quad \text{in} \; \Omega.$$
\end{lm}

Let $G(\cdot) \in \Phi(\Omega)$  strictly convex and satisfy $(SC)$, $(A_0)$,  $(A_1)$ and $(A_{1,n})$. Given a Sobolev ${G(\cdot)}$-regular subdomain $D \subset \Omega$ (see Corollary $4.1$) and $v$ is $G(\cdot)$-superharmonic fonction in $\Omega$.
Since $v$ is lower semicontinuous in $\Omega$, there exists a sequence $v_i \in C^\infty(\Omega)$ such that 
$$v_1 \leq v_2 \leq ... \leq v \; \text{and} \; \lim_{i\rightarrow\infty} v_i(x) = v(x)\; \text{at each} \;x \in \Omega.$$ 
Let $h_i$ be the $G(\cdot)$-harmonic function in $D$ such that $h_i - v_i \in W_0^{1,G(\cdot)}(D)$. Applying the Sobolev ${G(\cdot)}$-regularity of $D$ and the comparison principle, we get 
$$h_1 \leq h_2 \leq ... \leq v \; \text{in} \; D.$$
By the Harnack convergence theorem, the function $h = \lim_{i \rightarrow \infty}h_i$ is $G(\cdot)$-harmonic. 
We define the Poisson modification $P(v,D)$ as follows
$$P(v,D)= \left \{
   \begin{array}{r c l}
    h & \; \text{in} & D\\
    v & \; \text{in} & \Omega \backslash D.
   \end{array}
   \right.$$

\begin{rma}
If $v \in W^{1,G(\cdot)}(\Omega)$, then the Poisson modification of $v$ is defined as follows 
$$P(v,D)= \left \{
   \begin{array}{r c l}
    h & \; \text{in} & D\\
    v & \; \text{in} & \Omega \backslash D    
   \end{array}
   \right.$$
where $h$ is the $G(\cdot)$-harmonic function in $D$ such that $h-v \in W_0^{1,G(\cdot)}(D)$.
\end{rma}   

\begin{Th}
Let $G(\cdot) \in \Phi(\Omega)$  strictly convex and satisfy $(SC)$, $(A_0)$,  $(A_1)$ and $(A_{1,n})$. Let $D \subset \Omega$ be a $G(\cdot)$-regular subdomain and $v$ is a $G(\cdot)$-superharmonic function in $\Omega$. Then the Poisson modification $P(v,D)$ is $G(\cdot)$-superharmonic function in $\Omega$, $G(\cdot)$-harmonic function in $D$ and $P(v,D) \leq v.$.
\end{Th}

\begin{proof}
By the construction of the Poisson modification, we have $P(v,D)$ is a $G(\cdot)$-harmonic function in $D$, and $h \leq v$ in $D$, so
$$P(v,D) \leq v \; \text{in} \; \Omega.$$ 
We show that $P(v,D)$ is lower semicontinuous. Let $\xi \in \partial D$
$$ \liminf_{\substack{ x\to\xi \\ x\in\Omega \backslash D}}P(v,D) = \liminf_{\substack{ x\to\xi \\ x\in\Omega \backslash D}} v(x) \geq v(\xi) = P(v,D)(\xi)$$
and
$$\liminf_{\substack{ x\to\xi \\ x\in D}} P(v,D)(x) = \liminf_{\substack{ x\to\xi \\ x\in D}} h(x) \geq \liminf_{\substack{ x\to\xi \\ x\in D}} h_i(x) = v_i(\xi).$$
So,
$$\liminf_{\substack{ x\to\xi \\ x\in D}} P(v,D)(x)\geq v(\xi) = P(v,D)(\xi).$$
Next, we prove $P(v,D)$ satisfies the comparison principle. Indeed, let $G \subset \subset \Omega$ is a domain and $H \in C(\overline{G})$ is $G(\cdot)$-harmonic function in $G$ with $\restr{H}{\partial G} \leq \restr{P(v,D)}{\partial G}$.\\
We have $P(v,D) \leq v$ in $\Omega$, then $\restr{H}{\partial G} \leq \restr{v}{\partial G}$. As $v$ is $G(\cdot)$-superharmonic function, then $H \leq v$ in $G$. Hence,
$$H \leq P(v,D) \; \text{in} \; G \backslash D.$$
Let $\xi \in \partial (G \cap D)$, we have
$$ H(\xi) \leq v(\xi) \leq \liminf_{\substack{ x\to\xi \\ x\in D \cap G}} h(x).$$
So
$$\liminf_{\substack{ x\to\xi \\ x\in D \cap G}}H(x) \leq \liminf_{\substack{ x\to\xi \\ x\in D \cap G}}h(x).$$
Then
$$H \leq h  = P(v,D) \;\text{in} \; D \cap G.$$
Hence
$$H \leq P(v,D) \; \text{in} \; G.$$\\
Therefore  $P(v,D)$ is $G(\cdot)$-superharmonic function in $\Omega$.
\end{proof}

\subsection{$G(\cdot)$-resolutivity}

\begin{df}
Let $G(\cdot) \in \Phi(\Omega)$  strictly convex and satisfy $(SC)$, $(A_0)$,  $(A_1)$ and $(A_{1,n})$. We say that a function $f: \Omega \to [-\infty, \infty]$ is $G(\cdot)$-resolutive if the upper and the lower Perron $G(\cdot)$-solution $\overline{H}_f$ and $\underline{H}_f$ coincide and are $G(\cdot)$-harmonic in $\Omega$.
\end{df}

\begin{df}
A family $U$ of functions is down ward directed if for each $u,v \in U$, there is $s \in U$ with $s\leq \min(u,v)$
\end{df}

The following Lemma is fundamental in PWB method $\cite{ref14}$. The first recall that the lower semicontinuous regularization $u^*$ of any function $u: \Omega \to [-\infty, \infty]$ is defined by
$$u^*(x) := \lim_{r \to 0} \inf_{\Omega \cap B(x,r)}u.$$

\begin{lm}[Choquet's topological lemma]
Suppose $E \subset \R^N$ and that $U =\{ u_\gamma , \gamma \in I\}$ is a family of functions $u_\gamma: E \rightarrow [-\infty, \infty]$. Let $u = \inf U$. If $U$ is down ward directed, then there is a decreasing  sequence of functions $v_j \in U$ with limit $v$ such that the lower semicontinuous  regularizations $u^*$ and $v^*$ coincide.
\end{lm}

\begin{Th}
Let $G(\cdot) \in \Phi(\Omega)$  strictly convex and satisfy $(SC)$, $(A_0)$,  $(A_1)$ and $(A_{1,n})$. Then one of the following alternatives is true
\begin{enumerate}
\item[i)] $\overline{H}_f$ is $G(\cdot)$-harmonic in $\Omega$,
\item[ii)] $\overline{H}_f \equiv -\infty$,
\item[iii)] $\overline{H}_f \equiv \infty$.
\end{enumerate}
A similar statement is true for $\underline{H}_f$.
\end{Th}

\begin{proof}
If the upper class $U_f$ is empty, then $\overline{H}_f = \infty$.\\
Suppose that the upper class $U_f$ is not empty, then $U_f$ is down ward directed. So, by Choquet's topological lemma, there exists a decreasing sequence of functions $u_i \in U_f$ convergent to a function $u$ such that $u^* = \overline{H}_f$ in $\Omega$.\\ 
Let $D \subset \subset \Omega$ is a Sobolev $G(\cdot)$-regular and consider the Poisson modification $P(u_i,D)$. Using Theorem $5.2$, we have $P(u_i,D) \in U_f$. Then, by the Harnack convergence theorem,  $\lim_{i \rightarrow \infty}P(u_i,D)$ is either $G(\cdot)$-harmonic or identically $-\infty$ in $D$. As $\overline{H}_f \leq P(u_i,D) \leq u_i$ and $u^* = \overline{H}_f$, then $\overline{H}_f = \lim_{i \rightarrow \infty}P(u_i,B)$ in $D$. Therefore $\overline{H}_f$ is either $G(\cdot)$-harmonic or identically $-\infty$ in $\Omega$.
\end{proof}

\begin{Th}[Wiener theorem]
Let $G(\cdot) \in \Phi(\Omega)$  strictly convex and satisfy $(SC)$, $(A_0)$,  $(A_1)$ and $(A_{1,n})$. Suppose that $f: \partial \Omega \to \R$ is continuous. Then $f$ is $G(\cdot)$-resolutive in $\Omega$, i.e $\overline{H}_f = \underline{H}_f := H_f$.
\end{Th}

\begin{proof}
Let $f: \partial \Omega \to \R$ is a continuous function. By the Tietze extension theorem, we can assume $f \in C(\R^n)$, then there exists  $\varphi_i \in C^\infty(\R^n)$ such that for all $\epsilon > 0$, we have
$$ \varphi_i(\xi) - \epsilon < f(\xi) < \varphi_i(\xi) + \epsilon \; \; \text{when} \; \; \xi \in \partial \Omega.$$  
Thus, 
$$\underline{H}_{\varphi_i} - \epsilon \leq \underline{H}_{\varphi_i-\epsilon} \leq \underline{H}_f \leq \overline{H}_f\leq \overline{H}_{\varphi_i + \epsilon} \leq \overline{H}_{\varphi_i} + \epsilon. $$
So, if $\underline{H}_{\varphi_i} = \overline{H}_{\varphi_i}$, then $\underline{H}_f=\overline{H}_f$. Hence, it suffices to prove the result for $\varphi_i$. \\
Let $H_i$ be a $G(\cdot)$-harmonic in $\Omega$ such that  $H_i - \varphi_i \in W^{1, G(\cdot)}_0(\Omega)$. Let $v_i$ denote the $G(\cdot)$-solution to obstacle problem with $\varphi_i$ acting as obstacle and also boundary data. So $v_i \in U_f$. Choose Sobolev $G(\cdot)$-regular domains $D_j \subset  \subset \Omega$ such that $\Omega = \cup_{j \geq 1}D_j$ and $D_1 \subset D_2 \subset ...$. Construct the sequence of Poisson modification
$$P_{i,j} = P(v_i,D_j)$$
Then $\{P_{i,j}\}_j$ is non-increasing, $P_{i,j} \in U_f$ and $P_{i,j} - \varphi_i$. Then $P_{i,j} - \varphi_i = P_{i,j} - v_i + v_i - \varphi_i \in W^{1, G(\cdot)}_0(\Omega)$. Let $P_i = \lim_{j\to\infty}P_{i,j}$. As $\overline{H}_{\varphi_i} \leq P_{i,j}$, then by the Harnack convergence theorem $P_i$ is $G(\cdot)$-harmonic in $\Omega$ and $P_i - \varphi_i \in W^{1, G(\cdot)}_0(\Omega)$. So, $P_i = H_i$ in $\Omega$. Hence
$\overline{H}_{\varphi_i} \leq P_i = H_i$. By a similar proof, we have $H_i \leq \underline{H}_{\varphi_i}$. Then
$$H_i \leq \underline{H}_{\varphi_i} \leq \overline{H}_{\varphi_i} \leq H_i.$$
Hence $$\underline{H}_{\varphi_i} = \overline{H}_{\varphi_i}.$$
This concludes the proof.
\end{proof}

As a consequence of the previous theorem, the Perron $G(\cdot)$-solution coincides with the $G(\cdot)$-solution of Dirichlet-Sobolev with boundary $f$.

\begin{cor}
Let $G(\cdot) \in \Phi(\Omega)$  strictly convex and satisfy $(SC)$, $(A_0)$,  $(A_1)$ and $(A_{1,n})$. If $f \in W^{1, G(\cdot)}(\Omega)\cap C(\overline{\Omega})$. Then $\overline{H}_f$ is the unique $G(\cdot)$-harmonic function such that $\overline{H} - f \in W_0^{1, G(\cdot)}(\Omega)$.
\end{cor}

\section{${G(\cdot)}$-potential}

\begin{df}
Let $G(\cdot) \in \Phi(\Omega)$  strictly convex and satisfy $(SC)$, $(A_0)$,  $(A_1)$ and $(A_{1,n})$. Let $K \subset B$ be compact and $\psi \in C^\infty_0(B)$ be such that $\psi = 1$ on $K$. We define the ${G(\cdot)}$-potential for $K$ with respect to $B$ as follows
$$\mathcal{R}_{G(\cdot)}(K,B) := \left \{
   \begin{array}{r c l}
   h & \; \text{in} & B \backslash K\\
   1 & \text{in} & K
   \end{array}
   \right.$$
where $h$ is the unique $G(\cdot)$-harmonic function in $B \backslash K$ such that $h - \psi \in W_0^{1, G(\cdot)}(B \backslash K)$.
\end{df}

\begin{rma}
The definition of $\mathcal{R}_{G(\cdot)}(K,B)$ is independent of the particular choice of $\psi$. Indeed, if $\tilde{\psi}$ is another such that $\tilde{h}$ is the unique $G(\cdot)$-harmonic function in $B \backslash K$ such that $\tilde{h} - \tilde{\psi} \in W_0^{1, G(\cdot)}(B \backslash K)$, then $h - \tilde{h} \in W_0^{1, G(\cdot)}(B \backslash K)$ and by the uniqueness we have $h = \tilde{h}$ in $ W_0^{1, G(\cdot)}(B \backslash K)$. 
\end{rma}

\subsection{${G(\cdot)}$-potential and ${G(\cdot)}$-capacity}

Using the same method, as in $\cite{ref10}$, we prove the following lemma.

\begin{lm}
Let $G(\cdot) \in \Phi(B) \cap C^1(\R^+)$ satisfies $(SC)$ and $u = \mathcal{R}_{G(\cdot)}(K,B)$ the ${G(\cdot)}$-potential for $K$ with respect to $B$. Then $u$ is a $G(\cdot)$-supersolution in $B$.
\end{lm}

\begin{proof}
Let $G(\cdot) \in \Phi(\Omega) \cap C^1(\R^+)$. In $\cite{ref21}$ we have $u$ is a $G(\cdot)$-supersolution in $B$ equivalently
$$\int_{B} G(x, |\nabla u|) \, \mathrm{d}x \leq \int_{B} G(x, |\nabla (u + \varphi)|) \, \mathrm{d}x$$
for every nonnegative function $\varphi$ in $ W_0^{1, G(\cdot)}(B)$. So, we can assume that $u+\varphi \leq 1$, a.e in $B$. As $u = 1$ in $K$, then the inequality $u + \varphi \leq 1$ a.e. implies that $\varphi = 0$ a.e. on $K$. Hence $\varphi \in W_0^{1, G(\cdot)}(B \backslash K)$. Since $u$ is a $G(\cdot)$-harmonic function in $B \backslash K$, then
$$\int_{B} G(x, |\nabla u|) \, \mathrm{d}x = \int_{B \backslash K} G(x, |\nabla u|) \, \mathrm{d}x   \leq \int_{B \backslash K} G(x, |\nabla (u + \varphi)|) \, \mathrm{d}x \leq \int_{B} G(x, |\nabla (u + \varphi)|) \, \mathrm{d}x.$$
Therefore $u$ is a $G(\cdot)$-supersolution in $B$.
\end{proof}

Using the Riesz representation theorem, we have the following theorem.

\begin{lm}
Let $G(\cdot) \in \Phi(\Omega)$. For every $G(\cdot)$-supersolution $u$ in $\Omega$, there is a Radon measure $\mu[u] \in  \left( W^{1,G(\cdot)}_0(\Omega)\right)^*$ such that 
$$\int_{\Omega} \frac{g(x,|\nabla u|)}{|\nabla u|} \nabla u \cdot \nabla \varphi \, \mathrm{d}x = \int_{\Omega} \varphi \, \mathrm{d}\mu[u]$$
whenever $ \varphi \in  W^{1,G(\cdot)}_0(\Omega)$.
\end{lm}

\begin{Th}
Let $G(\cdot) \in \Phi(B) \cap C^1(\R^+)$ satisfies $(SC)$ and $K$ be a compact subset of $B$. If $u = \mathcal{R}_{G(\cdot)}(K,B)$ is the ${G(\cdot)}$-potential for $K$ with respect to $B$ and $\mu[u]$ its associated Radon measure in $W^{1, G(\cdot)}_0(B)^*$, then there exists a constant $C>0$ such that
$$ \frac{1}{C} \text{Cap}_{G(\cdot)}(K;B) \leq \mu[u](K) \leq C\text{Cap}_{G(\cdot)}(K;B)$$
\end{Th}

\begin{proof}
Let $u$ the ${G(\cdot)}$-potential for $K$ with respect to $B$ and $\mu[u]$ its associated Radon measure in $W^{1, G(\cdot)}_0(B)^*$. As $u$ is ${G(\cdot)}$-harmonic in $B \backslash K$, then the support of the measure $\mu[u]$ is contained in $K$. Hence
\begin{equation}
\mu[u](K) =\mu[u](B) = \displaystyle \int_{B} u \, \mathrm{d}\mu[u] = \displaystyle \int_{B} \frac{g(x,|\nabla u|)}{|\nabla u|} \nabla u \cdot \nabla u \, \mathrm{d}x.
\end{equation}
On the one hand, as $u \in S_{G(\cdot)}(K;\Omega)$ then 
$$\text{cap}_{G(\cdot)}(K;B) \leq \displaystyle \int_{B} G(x,|\nabla u|) \, \mathrm{d}x \leq C\displaystyle \int_{B} \frac{g(x,|\nabla u|)}{|\nabla u|} \nabla u \cdot \nabla u \, \mathrm{d}x \leq C \mu[u](K)$$
On the other hand, let $\varphi \in S_{G(\cdot)}(K;\Omega)$ and we consider $\psi = \max\{\varphi-u,u\}$, so the nonnegative function $\psi - u \in W^{1, G(\cdot)}_0(B)$. Since $u$ is a $G(\cdot)$-supersolution, we have
$$ \displaystyle \int_{B} \frac{g(x,|\nabla u|)}{|\nabla u|} \nabla u \cdot \nabla (\varphi - u) \, \mathrm{d}x \geq \displaystyle \int_{B} \frac{g(x,|\nabla u|)}{|\nabla u|} \nabla u \cdot \nabla (\psi - u) \, \mathrm{d}x \geq 0.$$
Then
$$\displaystyle \int_{B} \frac{g(x,|\nabla u|)}{|\nabla u|} \nabla u \cdot \nabla u \, \mathrm{d}x \leq \displaystyle \int_{B} \frac{g(x,|\nabla u|)}{|\nabla u|} \nabla u \cdot \nabla \varphi \, \mathrm{d}x.$$
Using the inequality $(2.1)$, we get
$$\begin{array}{ll}
\displaystyle \int_{B} G(x,|\nabla u|) \, \mathrm{d}x & \leq C \displaystyle \int_{B} g(x,|\nabla u|)|\nabla \varphi| \, \mathrm{d}x\\
& \leq \displaystyle \frac{1}{2} \int_{B} G(x,|\nabla u|) \, \mathrm{d}x + C  \int_{B} G(x,|\nabla \varphi|) \, \mathrm{d}x.
\end{array}$$
Hence
$$\int_{B} G(x,|\nabla u|) \, \mathrm{d}x \leq C\int_{B} G(x,|\nabla \varphi|) \, \mathrm{d}x. $$
By the equality $(6.1)$, we have
$$\mu[u](K) \leq  C\int_{B} G(x,|\nabla u|) \, \mathrm{d}x \leq C\int_{B} G(x,|\nabla \varphi|) \, \mathrm{d}x.$$
Taking the infimum of the functions $\varphi \in S_{G(\cdot)}(K;B)$, we obtain
$$ \mu[u](K) \leq C\text{cap}_{G(\cdot)}(K;B).$$
This concludes the proof.
\end{proof}

\subsection{Estimation of ${G(\cdot)}$-potential}

In $\cite{ref4}$, we proved the following Caccioppoli type estimate of supersolutions to equation $(3.1)$.

\begin{lm}
Let $G(\cdot) \in \Phi(2B)$ satisfies $(SC)$. Let $u$ be a nonpositive $G(\cdot)$-supersolution of $(3.1)$ in a ball $2B$, $\eta \in C^\infty_0(2B)$ with $0 \leq \eta \leq 1$ and $|\nabla \eta| \leq \lfrac{1}{r} $. Then, there exits a constant $C$ such that
$$\displaystyle \int_{2B} G(x, |\nabla u|) \eta^{g^0} \, \mathrm{d}x \leq C \displaystyle \int_{2B} G^+\left(\frac{-u}{r}\right) \, \mathrm{d}x.$$
\end{lm}

\begin{lm}
Let $G(\cdot) \in \Phi(B(x_0, 2r))$ satisfy $(SC)$, $(A_0)$ and $(A_{1,n})$. If $u$ is a nonnegative $G(\cdot)$-supersolution in $B(x_0, 2r)$, then, for some constant $C>0$, we have
$$C r g^{-1}\left( x_0,\frac{\mu[u](B(x_0, r))}{R^{n-1}} \right) \leq \essinf_{B(x_0, r)}u + R,$$
with $\mu[u]$ is the associated Radon measure to $u$ in $\left( W^{1, G(\cdot)}_0(B(x_0, 2r))\right)^*$.
\end{lm}

\begin{proof}
We set $B = B(x_0, r)$, $b = \inf_{B} u$ and, $v = \min\{u,b\} + r$. Choose $\omega = v\eta^{g^0} $ such that $\eta \in C^\infty_0(2B)$ with $0 \leq \eta \leq 1$, $\eta = 1$ in $\overline{B}$ and $|\nabla \eta| \leq \displaystyle \frac{C}{r} $, we have
$$\begin{array}{ll}
(b+r)\mu[u](B) & \leq  \displaystyle\int_{2B}  \omega \, \mathrm{d}\mu[u] \\
& =  \displaystyle \int_{2B} \frac{g(x,|\nabla u|)}{|\nabla u|} \nabla u \cdot \nabla \omega \, \mathrm{d}x \\
&\leq  \displaystyle \int_{2B} \left(\frac{g(x,|\nabla u|)}{|\nabla u|} \nabla u \cdot \nabla v\right) \eta^{g^0} \, \mathrm{d}x + \displaystyle \int_{2B} \left( \frac{g(x,|\nabla u|)}{|\nabla u|} \nabla u \cdot \nabla \eta\right)\eta^{g^0-1}v \, \mathrm{d}x.\\
\end{array}$$
By the condition $(SC)$, we have
$$\begin{array}{ll}
I_1 & := \displaystyle \int_{2B} \left(\frac{g(x,|\nabla u|)}{|\nabla u|} \nabla u \cdot \nabla v\right)\eta^{g^0} \, \mathrm{d}x\\
& \leq g^0 \displaystyle \int_{2B} G(x, |\nabla v)|) \eta^{g^0} \, \mathrm{d}x
\end{array}$$
and
$$\begin{array}{ll}
I_2 & := \displaystyle \int_{2B} \left(\frac{g(x,|\nabla u|)}{|\nabla u|} \nabla u \cdot \nabla \eta\right)\eta^{g^0-1}v \, \mathrm{d}x\\
& \leq \displaystyle \int_{2B} g(x, |\nabla v|)|\nabla \eta|\eta^{g^0-1}v \, \mathrm{d}x.
\end{array}$$
As $v \leq b + r$ and $|\nabla \eta| < \displaystyle \frac{C}{r}$, we have
$$I_2 \leq C \displaystyle \frac{b+r}{r} \displaystyle \int_{2B} g(x, |\nabla v|)\eta^{g^0-1} \, \mathrm{d}x.$$
Using inequality $(2.1)$ for $a^\prime = |\nabla v|$ and $b^\prime = \displaystyle \frac{b+r}{\eta r} $, and the condition $(SC)$, we get
$$I_2 \leq C \left( \displaystyle \int_{2 B} G(x, |\nabla v|)\eta^{g^0} \, \mathrm{d}x + \displaystyle \int_{2B} G\left( x, \frac{b+r}{r}\right) \, \mathrm{d}x \right).$$
Collecting the previous estimations of $I_1$ and $I_2$, we obtain
$$(b+r)\mu[u](B) \leq C\left( \displaystyle \int_{2B} G(x, |\nabla v|)\eta^{g^0} \, \mathrm{d}x + \displaystyle \int_{2B} G\left( x, \frac{b+r}{r}\right) \, \mathrm{d}x \right).$$
Or, by Lemma $6.3$, we have
$$\displaystyle \int_{B} G(x,|\nabla( v-(b+R))|)\eta^{g^0} \, \mathrm{d}x \leq C \displaystyle \int_{2B} G^+\left(\frac{b+r-v}{r}\right) \, \mathrm{d}x.$$
Hence
$$ (b+r)\mu[u](B)\leq C \displaystyle \int_{2B} G^+\left(\frac{b+r}{r}\right) \, \mathrm{d}x.$$
Since $L^{G(\cdot)}(B) \subset L^{g_0}(B)$ (see $\cite{ref12}$), we have
$$1 \leq \displaystyle \frac{b + 1}{r} \leq \displaystyle \frac{\frac{\norme{u}_{g_0,B}}{|B|^{\frac{1}{g^0}}}+1}{r}.$$
Then, by Lemma  $2.1$, there exists a constant $C>0$ dependent of $\displaystyle\frac{\norme{u}_{g_0,B}}{|B|^{\frac{1}{g^0}}}$ such that
$$G^+\left(\frac{b+r}{r}\right) \leq C G\left( x_0,\frac{b+r}{r}\right). $$
Hence
$$ (b+r)\mu[u](B) \leq C r^{n} G\left( x_0,\frac{b+r}{r}\right).$$
So, by the condition $(SC)$, we have
$$ \mu[u](B) \leq C r^{n-1} g\left( x_0,\frac{b+r}{r}\right).$$
From inequalities $2.4$, $2.2$ and $2.3$, we have
$$C r g^{-1}\left( x_0,\frac{\mu[u](B)}{r^{n-1}} \right) \leq \inf_{B}u + r.$$
This concludes the proof.
\end{proof}

By a similar proof in $\cite{ref8}$, we have the following lemma. 

\begin{lm}
Let $G(\cdot) \in \Phi(\Omega)$ satisfies $(SC)$. Let $f \in W^{1, G(\cdot)}(\Omega)$ and $v$ be a $G(\cdot)$-supersolution in $\Omega$ such that $f-v \in W_0^{1, G(\cdot)}(\Omega)$. Then the solution of the obstacle problem with the obstacle $v$ and the boundary data $f$ is a $G(\cdot)$-solution in $\Omega$.
\end{lm}

\begin{Th}
Let $x_0 \in \partial \Omega$. Let $ G(\cdot) \in \Phi(\R^n) \cap C^1(\R^+)$ be strictly convex and satisfy $(SC)$, $(A_0)$, $(A_1)$, and $(A_{1,n})$. Fix $r > 0$, and let $u = \mathcal{R}_{G(\cdot)}(\overline{B}(x_0,r) \backslash \Omega , B(x_0,4r))$ be the $G(\cdot)$-potential for $\overline{B}(x_0,r) \backslash \Omega$ with respect to $B(x_0,4r)$. Then for $0 < \rho \leq r$ and $x \in B(x_0,\rho)$, we have
$$1-u(x) \leq \exp\left(-C\int_{\rho}^r g^{-1}\left(x_0 , \frac{\text{cap}_{G(\cdot)}(\overline{B}(x_0,t)\cap \Omega^\complement, B(x_0,2t)}{t^{n-1}}\right) \, \mathrm{d}t + Cr\right).$$
\end{Th}

\begin{proof}
 Let $x_0 \in \partial \Omega$, $r > 0$, and $B_j = B(x_0,r_j)$ where $r_j = 4^{1-j}r, \; j= 0, 1, 2, ...$. Let $u$ be the ${G(\cdot)}$-potential for $\overline{B_1} \cap \Omega^\complement$ with respect to $B_0$. By Lemma $6.4$, we have
$$\begin{array}{ll} 
m_1:=  \inf_{\frac{1}{2} B_0} u & \geq C\displaystyle\lfrac{r_0}{2} g^{-1}\left( x_0,\lfrac{\mu[u](\frac{1}{2}B_0)}{\left(\lfrac{r_0}{2}\right)^{n-1}}\right) - \displaystyle\lfrac{r_0}{2}\\
& \geq Cr_0 g^{-1}\left( x_0,\lfrac{\mu[u](\overline{B_1} \cap \Omega^\complement)}{r_0^{n-1}}\right) - \displaystyle\lfrac{r_0}{2}.
\end{array}$$
Using Theorem $6.1$, we get
\begin{equation}
m_1 \geq  Cr_0g^{-1}\left( x_0,\displaystyle \frac{\text{cap}_{G(\cdot)}(\overline{B_1} \cap \Omega^\complement; B_0)}{r_0^{n-1}}\right) - \displaystyle\lfrac{r_0}{2}.
\end{equation}
As $1+t \leq e^t$, then
\begin{equation}
\begin{array}{ll}
1-m_1 & \leq 1 - Cr_0g^{-1}\left( x_0,\displaystyle \frac{\text{cap}_{G(\cdot)}(\overline{B_1} \cap \Omega^\complement; B_0)}{r_0^{n-1}}\right) + \displaystyle\lfrac{r_0}{2} \\
& \leq \exp\left( - Cr_0g^{-1}\left( x_0,\displaystyle \frac{\text{cap}_{G(\cdot)}(\overline{B_1} \cap \Omega^\complement; B_0)}{r_0^{n-1}}\right) + \displaystyle\lfrac{r_0}{2}\right).
\end{array}
\end{equation}
Next, let $D_1 = B_1 \backslash (\overline{B_2} \cap \Omega^\complement)$ and let $f_1 \in W^{1, G(\cdot)}_0(B_0)$ such that $f_1= m_1$ on $\partial B_1$ and $f_1 = 1$ on $\overline{B_2}$.
Let $u_1$ be the solution of the obstacle problem in $D_1$ with the upper obstacle $u$ and the boundary values $f_1$ extend to $\overline{B_2} \cap \Omega^\complement$ by the constant $1$. Then $\displaystyle \frac{u_1 - m_1}{1 - m_1}$ is the ${G(\cdot)}$-potential for $\overline{B}_2 \cap \Omega^\complement$ with respect to $ B_1$. So, by inequality $(6.2)$, we have
$$ \inf_{\frac{1}{2} B_1}  \frac{u_1 - m_1}{1 - m_1} \geq Cr_1g^{-1}\left( x_0,\frac{\text{cap}_{G(\cdot)}(\overline{B_2} \cap \Omega^\complement; B_1)}{r_1^{n-1}}\right) - \frac{r_1}{2}.$$
Hence
$$m_2 := \inf_{\frac{1}{2} B_1} u_1 \geq Cr_1(1-m_1)g^{-1}\left( x_0,\frac{\text{cap}_{G(\cdot)}(\overline{B_2} \cap \Omega^\complement; B_1)}{r_1^{n-1}}\right) - \frac{r_1}{2}(1-m_1) + m_1.$$
Consequently
$$\begin{array}{ll}
1-m_2 & \leq  -  Cr_1(1-m_1)g^{-1}\left( x_0,\displaystyle \frac{\text{cap}_{G(\cdot)}(\overline{B_2} \cap \Omega^\complement; B_1)}{r_1^{n-1}}\right) + (1 + \displaystyle\lfrac{r_1}{2})(1 - m_1)\\
& \leq (1-m_1)\left( 1 -  Cr_1(1-m_1)g^{-1}\left( x_0,\displaystyle \frac{\text{cap}_{G(\cdot)}(\overline{B_2} \cap \Omega^\complement; B_1)}{r_1^{n-1}}\right) + \displaystyle\lfrac{r_1}{2}\right).
\end{array}$$
Then 
$$1-m_2 \leq  (1-m_1)\exp \left(- Cr_1g^{-1}\left( x_0,\frac{\text{cap}_{G(\cdot)}(\overline{B_2} \cap \Omega^\complement; B_1)}{r_1^{n-1}}\right) +  \displaystyle\lfrac{r_1}{2} \right).$$
A similar method, let $D_j = B_j \backslash (\overline{B_{j+1}} \cap \Omega^\complement)$ and let $f_j \in W^{1, G(\cdot)}_0(B_{j-1})$ such that  $f_j = m_j$ on $\partial B_j$ and $f_j = 1$ on $\overline{B_{j+1}}$.
$$f_j = \left \{
   \begin{array}{r c l}
    m_j & \; \text{on} & \partial B_j\\
    1 & \; \text{on} & \overline{B}_{j+1}. 
   \end{array}
   \right.$$
Let $u_j$ be the solution of the obstacle problem in $D_j$ with the upper obstacle $u_{j-1}$ and the boundary values $f_j$ extend to $\overline{B_{j+1}} \cap \Omega^\complement$ by the constant $1$. Then we have
$$1-m_{j+1} \leq  (1-m_j)\exp \left( - Cr_jg^{-1}\left( x_0,\frac{\text{cap}_{G(\cdot)}(\overline{B}_{j+1} \cap \Omega^\complement; B_j)}{r_j^{n-1}}\right) +  \displaystyle\lfrac{r_j}{2} \right).$$
with $m_{j+1} := \inf_{\frac{1}{2} B_j} u_j$. Iterating this inequality and using inequality $(6.3)$, we get for $k = 1,2, ... ,$
$$1-m_{k+1} \leq  \exp \left( -C\sum^k_{j=0} r_jg^{-1}\left( x_0,\frac{\text{cap}_{G(\cdot)}(\overline{B}_{j+1} \cap \Omega^\complement; B_j)}{r_j^{n-1}}\right) +  \sum^k_{j=0} \displaystyle\lfrac{r_j}{2} \right).$$ 
As $u \geq u_1$ and $u_j \geq u_{j+1}$in $B_{j+1}; \; j= 1,2, ... ,$, then
\begin{equation}
1-u \leq  \exp \left( -C\sum^k_{j=0} r_jg^{-1}\left( x_0,\frac{\text{cap}_{G(\cdot)}(\overline{B}_{j+1} \cap \Omega^\complement; B_j)}{r_j^{n-1}}\right) +  \sum^k_{j=0} \displaystyle\lfrac{r_j}{2} \right) \; \text{on} \; \frac{1}{2}\overline{B}_k.
\end{equation}
Fix $\rho > 0$ so that $\rho \leq r$ and choose an integer $k$ so that $r_{k+3} < \rho \leq r_{k+2}$, we have
$$\displaystyle\sum^k_{j=0} r_jg^{-1}\left( x_0, \displaystyle\frac{\text{cap}_{G(\cdot)}(\overline{B}_{j+1} \cap \Omega^\complement; B_j)}{r_j^{n-1}}\right) \geq C \displaystyle \sum^k_{j=0} \displaystyle \int_{r_{j+2}}^{r_{j+1}} g^{-1}\left( x_0,\displaystyle\frac{\text{cap}_{G(\cdot)}(\overline{B}_{j+1} \cap \Omega^\complement; B_j)}{t^{n-1}}\right) \, \mathrm{d}t.$$
Or using $r_{j+2} \leq t \leq r_{j+1}$ and Proposition $2.1$, we get
$$\begin{array}{ll}
\text{cap}_{G(\cdot)}(\overline{B}(x_0,t) \cap \Omega^\complement ; B(x_0,2t)) & \leq C\left(\text{cap}_{G(\cdot)}(\overline{B}(x_0,t) \cap \Omega^\complement; B(x_0,4t)) + t^n\right)\\
& \leq C\left(\text{cap}_{G(\cdot)}(\overline{B}(x_0,t) \cap \Omega^\complement; B(x_0,8t)) + t^{n}\right)\\
& \leq C\left(\text{cap}_{G(\cdot)}(\overline{B}_{j+1} \cap \Omega^\complement; B_{j}) + t^{n}\right).
\end{array}$$  
Then, we have
$$\begin{array}{ll}
 & \displaystyle \int_{r_{j+2}}^{r_{j+1}} g^{-1}\left( x_0,\displaystyle\frac{\text{cap}_{G(\cdot)}(\overline{B}_{j+1} \cap \Omega^\complement; B_j)}{t^{n-1}}\right) \, \mathrm{d}t\\
& = \displaystyle \int_{r_{j+2}}^{r_{j+1}} g^{-1}\left( x_0,\displaystyle\frac{\text{cap}_{G(\cdot)}(\overline{B}_{j+1} \cap \Omega^\complement; B_j) + t^{n-1}}{t^{n-1}} -1 \right) \, \mathrm{d}t\\
& \geq  C\displaystyle \int_{r_{j+2}}^{r_{j+1}} g^{-1}\left( x_0,\displaystyle\frac{\text{cap}_{G(\cdot)}(\overline{B}_{j+1} \cap \Omega^\complement; B_j) + t^{n-1}}{t^{n-1}}\right) \, \mathrm{d}t - \int_{r_{j+2}}^{r_{j+1}} g^{-1}(x_0,1) \, \mathrm{d}t \\
& \geq  C\displaystyle \int_{r_{j+2}}^{r_{j+1}} g^{-1}\left( x_0,\displaystyle\frac{\text{cap}_{G(\cdot)}(\overline{B}_{j+1} \cap \Omega^\complement; B_j) + t^{n}}{t^{n-1}}\right) \, \mathrm{d}t - \int_{r_{j+2}}^{r_{j+1}} g^{-1}(x_0,1) \, \mathrm{d}t \\
& \geq  C \displaystyle \int_{r_{j+2}}^{r_{j+1}} g^{-1}\left( x_0,\displaystyle\frac{\text{cap}_{G(\cdot)}(\overline{B}(x_0,t) \cap \Omega^\complement; \overline{B}(x_0,2t))}{t^{n-1}}\right) \, \mathrm{d}t - \int_{r_{j+2}}^{r_{j+1}} g^{-1}(x_0,1) \, \mathrm{d}t.
\end{array}$$
Hence,  by the condition $(A_0)$, we obtain
$$\begin{array}{ll}
\displaystyle\sum^k_{j=0} r_jg^{-1}\left( x_0, \displaystyle\frac{\text{cap}_{G(\cdot)}(\overline{B}_{j+1} \cap \Omega^\complement; B_j)}{t^{n-1}}\right) \geq C\displaystyle \int_{\rho}^{r} g^{-1}\left( x_0,\displaystyle\frac{\text{cap}_{G(\cdot)}(\overline{B}(x_0,t) \cap \Omega^\complement; \overline{B}(x_0,2t))}{t^{n-1}}\right) \, \mathrm{d}t - Cr.
\end{array}$$
Then, for $x \in B(x_0,\rho)$, we get
$$ 1-u(x) \leq \exp \left(-C\displaystyle \int_{\rho}^{r} g^{-1}\left( x_0,\displaystyle\frac{\text{cap}_{G(\cdot)}(\overline{B}(x_0,t) \cap \Omega^\complement; B(x_0,2t))}{t^{n-1}}\right) \, \mathrm{d}t + Cr\right).$$
This concludes the proof.
\end{proof}

\begin{Th}
Let $x_0 \in \partial \Omega$. Let $ G(\cdot) \in \Phi(\R^n) \cap C^1(\R^+)$ be strictly convex and satisfy $(SC)$, $(A_0)$, $(A_1)$, and $(A_{1,n})$. Fix $r > 0$ and let $u$ be the $G(\cdot)$-potential for $\overline{B}(x_0,r) \backslash \Omega$ with respect to $B(x_0,4r)$. Then
$$\liminf_{x \to x_0} u(x) \leq  C \left(\displaystyle \int_{0}^{4r} g^{-1}\left( x_0, \displaystyle \frac{\text{cap}_{G(\cdot)}(\overline{B}(x_0,t) \cap \Omega^\complement; B(x_0,2t))}{t^{n-1}} \right) \, \mathrm{d}t + r\right).$$ 
\end{Th}

\begin{proof}
Let $u$ be the $G(\cdot)$-potential for $\overline{B}(x_0,r) \backslash \Omega$ with respect to $B(x_0,4r)$. Then by the Wolff potential upper estimate Theorem $5.12$ in $\cite{ref6}$ and Theorem $4.4$ in $\cite{ref13}$, we have
$$\lim_{\rho \to 0} \inf_{\Omega \cap B(x_0,\rho)}u(x) \leq C\left( r + \inf_{B(x_0,2r)}u + \displaystyle \int_0^{4r} g^{-1}\left( x_0,\frac{\mu[u](B(x_0,t)}{t^{n-1}}\right) \, \mathrm{d}t \right).$$
Next, let  $0 < t \leq 4r$, $B = B(x_0,r)$, $\mu_t$ be the restriction of $\mu[u]$ to $B(x_0,t)$ and $u_t \in W^{1, G(\cdot)}_0(4B) $ be the ${G(\cdot)}$-supersolution in $4B$ associated with $\mu_t$. So we have 
$$\int_{4B} \frac{g(x,|\nabla u_t|)}{|\nabla u_t|}\nabla u_t \cdot \nabla \varphi \, \mathrm{d}x = \int_{4B} \varphi \, \mathrm{d}\mu_t \; \; \text{far all} \; \varphi \in W^{1, G(\cdot)}_0(4B).$$
As 
$$\int_{4B} \frac{g(x,|\nabla u|)}{|\nabla u|}\nabla u \cdot \nabla \varphi \, \mathrm{d}x = \int_{4B} \varphi \, \mathrm{d}\mu[u] \; \; \text{far all} \; \varphi \in W^{1, G(\cdot)}_0(4B).$$
Choosing $\varphi = (u_t-u)_+$ as a test function in the two previous inequalities, then
$$\int_{2B} \left(\left(\frac{g(x,|\nabla u_t|)}{|\nabla u_t|}\nabla u_t - \frac{g(x,|\nabla u|)}{|\nabla u|}\nabla u\right) \cdot (\nabla u_t - \nabla u)\right) \, \mathrm{d}x = 0.$$
Hence $\nabla(u_t-u)=0$ a.e in $4B$, then $u_t \leq u \leq 1$ a.e in $4B$. So, by Theorem $6.1$ and Proposition $2.1$, we have
$$\mu_t(B(x_0,t)) \leq C\text{cap}_{G(\cdot)}(\overline{B}(x_0,t) \cap \Omega^\complement; 4B) \leq C\text{cap}_{G(\cdot)}(\overline{B}(x_0,2t) \cap \Omega^\complement; B(x_0,4t)).$$
Let $\lambda = \inf_{2B}u$ and  $B(y,\lfrac{r}{4}) \subset B \cap \Omega^{\complement}$, so by the condition $(SC)$, we get
$$\begin{array}{ll}
r^{n-1} g\left( x_0, \displaystyle \frac{\lambda}{r}\right) & \displaystyle \leq C\left|B\left( y,\frac{r}{4}\right)\right| G\left( x_0,\lfrac{1}{r}\right)\\
& \leq C  \displaystyle \int_{B(y,\frac{r}{4})} G\left( x_0,\frac{u}{r}\right) \, \mathrm{d}x\\
& \leq C  \displaystyle \int_{4B} G\left( x_0,\frac{u+r}{r}\right) \, \mathrm{d}x\\
\end{array}$$

As $1 \leq \displaystyle \frac{u+r}{r} \leq \displaystyle \frac{2}{r}$ then, by Lemma $2.1$, we have
$$\displaystyle \int_{4B} G\left( x_0,\frac{u+r}{r}\right) \, \mathrm{d}x \leq \displaystyle \int_{4B} G\left(x,\frac{u+r}{r}\right) \, \mathrm{d}x $$
Then, using the Poincaré inequality and the condition $(A_0)$, we obtain
$$\begin{array}{ll}
r^{n-1} g\left( x_0, \displaystyle \frac{\lambda}{r}\right) & \leq C \left( \displaystyle \int_{4B} G\left( x,\frac{u}{r}\right) \, \mathrm{d}x + \displaystyle \int_{4B} G\left( x,1\right) \, \mathrm{d}x \right)\\
& \leq C  \left(\displaystyle \int_{4B} G(x,|\nabla u|) \, \mathrm{d}x + r^n + G(x_0,1)|4B| \right)\\
& \leq  C  \left(\displaystyle \int_{4B} G(x,|\nabla u|) \, \mathrm{d}x + r^{n} \right).
\end{array}$$
Or, from Lemma $6.2$, if we choose $\varphi = u$, we obtain
$$\displaystyle \int_{4B} G(x,|\nabla u|) \, \mathrm{d}x \leq C \int_{4B} \frac{g(x,|\nabla u|)}{|\nabla u|} \nabla u \cdot \nabla u \, \mathrm{d}x = C \int_{4B} u \, \mathrm{d}\mu[u] \leq C\mu[u](4B).$$
Then
$$r^{n-1} g\left( x_0, \displaystyle \frac{\lambda}{r}\right) \leq C \left( \mu[u](4B) + r^n\right). $$
From Theorem $6.1$, we have 
$$r^{n-1} g(x_0, \displaystyle \frac{\lambda}{r}) \leq C\left(\text{cap}_{G(\cdot)}(\overline{B}(x_0,r) \backslash \Omega ; 4B) + r^{n}\right).$$
Using inequalities $2.4$, $2.2$ and $2.3$, we get
$$ \lambda \leq C\left(rg^{-1}\left( x_0,\frac{\text{cap}_{G(\cdot)}(\overline{B}(x_0,r) \backslash \Omega ; 4B)}{r^{n-1}}\right) + r^2 \right).$$
Therefore
$$ \inf_{2B}u \leq C  \left( \displaystyle \int_{r}^{2r} g^{-1}\left( x_0,\displaystyle\frac{\text{cap}_{G(\cdot)}(\overline{B}(x_0,t) \cap \Omega^\complement; B(x_0,2t))}{t^{n-1}} \right)\, \mathrm{d}t + r^2 \right).$$
Hence
$$\liminf_{x \to x_0} u(x) \leq  C \left(\displaystyle \int_{0}^{4r} g^{-1}\left( x_0,\displaystyle\frac{\text{cap}_{G(\cdot)}(\overline{B}(x_0,t) \cap \Omega^\complement; B(x_0,2t))}{t^{n-1}} \right) \, \mathrm{d}t + r\right).$$
This concludes the proof.
\end{proof}

\section{Criterion Wiener}

First of all, the notion of the regularity of boundary points is defined in connection with Perron $G(\cdot)$-solutions.

\begin{df}
Let $G(\cdot) \in \Phi(\Omega)$. A boundary point $x_0$ of an open set $\Omega$ is called $G(\cdot)$-regular if 
$$\lim_{x\to x_0}\overline{H}_f(x) = f(x_0)$$
for each continuous $f : \partial \Omega \to \R$.
\end{df}

The following lemma shows that $G(\cdot)$-regularity is a local property.

\begin{lm}
Let $G(\cdot) \in \Phi(\Omega)$. A boundary $x_0$ of $\Omega$ is $G(\cdot)$-regular if and only if 
$$\lim_{x\to x_0}\overline{H}_f(x) = f(x_0)$$
for each bounded $f : \partial \Omega \to \R$, continuous at $x_0$. 
\end{lm}

\begin{proof}
Let $x_0 \in \partial \Omega$ be $G(\cdot)$-regular and fix $\epsilon > 0$. Let $U$ be an neighborhood of $x_0$ such that $|f-f(x_0)|  < \epsilon$ on $U \cap \partial \Omega$. Then, choose a continuous function $g: \partial \Omega \to [f(x_0)+\epsilon, \sup|f| + \epsilon]$ such that $g(x_0) = f(x_0) + \epsilon$ and $g = \sup|f|+\epsilon$ on $\partial \Omega \backslash U$. Now $g \geq f$ on $\partial \Omega$ and hence we have
$$\limsup_{x\to x_0}\overline{H}_f(x) \leq \lim_{x\to x_0}\overline{H}_g(x) = g(x_0) = f(x_0) + \epsilon.$$
Similarly, we have 
$$\liminf_{x\to x_0}\overline{H}_f(x) \geq f(x_0) - \epsilon.$$
Thus we conclude
$$\lim_{x \to x_0}\overline{H}_f(x) = f(x_0).$$
and the lemma is proved.
\end{proof}

\begin{lm}
Let $G(\cdot) \in \Phi(\Omega)$. Assume that $f: \partial \Omega \to \R$ is $G(\cdot)$-resolutive. Let $\Omega^\prime \subset \Omega$ be open and define $\tilde{f} : \partial \Omega^\prime \to \R$ by 
$$\tilde{f}(x) = \left \{
   \begin{array}{r c l}
    f(x) & \; \text{if} & x \in \partial \Omega \cap \partial \Omega^\prime\\
    H_f(x) & \; \text{if} & x \in \Omega \cap \partial \Omega^\prime.   
   \end{array}
   \right.$$
Then $\tilde{f}$ is $G(\cdot)$-resolutive with respect to $\Omega^\prime$ and the Perron $G(\cdot)$-solution for $\tilde{f}$ in $\Omega^\prime$ is $\restr{H_f}{\Omega^\prime}$
\end{lm}

\begin{proof}
Let $f: \partial \Omega \to \overline{R}$ be a $G(\cdot)$-resolutive,  $\Omega^\prime \subset \Omega$ and $ u \in U_f$. As $u$ is lower semicontinuous, then for each $y \in \Omega^\prime$
$$\lim_{y \to x} u(y) \geq \tilde{f}(x) \quad \text{for all} \, x \in \partial \Omega^\prime.$$
Hence $u \in U_{\tilde{f}}$ for $\tilde{f}$ in $\Omega^\prime$. So taking infimum over all $u$,we have
$$\overline{H}_{\tilde{f}} \leq \restr{H_f}{\Omega^\prime}.$$
Applying the same argument to $-f$, we obtain
$$ \underline{H}_{\tilde{f}} \leq \overline{H}_{\tilde{f}} \leq H_f = -H_{-f} \leq - \overline{H}_{\tilde{f}} \leq \underline{H}_{\tilde{f}} \quad \text{in} \; \Omega^\prime.$$
This concludes the proof.
\end{proof}

\begin{Th}
 Let $ G(\cdot) \in \Phi(\R^n) \cap C^1(\R^+)$ be strictly convex and satisfy $(SC)$, $(A_0)$, $(A_1)$, and $(A_{1,n})$. The point $x_0 \in \partial \Omega$ is ${G(\cdot)}$-regular if and only if for some $\rho > 0$, 
\begin{equation}
\int_0^{\rho} g^{-1}\left(x_0 , \frac{\textsl{cap}_{G(\cdot)}(B(x_0,t)\cap \Omega^\complement, B(x_0,2t)}{t^{n-1}}\right) \, \mathrm{d}t = \infty.
\end{equation}
\end{Th}

\begin{proof}
Let $f \in C(\partial \Omega)$ and $\epsilon > 0$ be arbitrary. There exists $r > 0$ such that 
$$\sup_{\partial \Omega \cap B(x_0, 2r)}|f-f(x)| \leq \epsilon.$$ 
Let $u$ be the $G(\cdot)$-potential for $\overline{B(x_0,r}) \backslash \Omega$ with respect to $B(x_0,4r)$ and $\tilde{f}$ be as in Lemme $7.2$ with $\Omega^\prime := \Omega \cap 4B$. So, we put $B=B(x_0,r)$ , $m=\sup_{\partial \Omega \cap 2B}(f-f(x_0))$ and $M = \sup_{\partial \Omega}(f-f(x_0))$. Then, we have 
$$h - f(x_0) \leq m + M(1-u) \;\text{on} \; \partial \Omega^\prime.$$
Using Lemma $5.1$ and Lemma $7.2$, we get 
$$H_f-f(x_0) = \restr{H_h}{\Omega^\prime}-f(x_0) \leq  \restr{H_{h - f(x_0)}}{\Omega^\prime} \leq H_{m+M(1-u)} = m+M(1-u) \; \text{on} \; \Omega^\prime. $$
Hence, from Theorem $6.2$, we have
$$\begin{array}{ll}
& \sup_{\Omega \cap B(x_0,\rho)}\left(H_f-f(x_0)\right) \leq  \sup_{\partial \Omega \cap 2B}\left(f-f(x_0)\right)\\
& + \sup_{\partial \Omega}\left(f-f(x_0)\right)\exp \left(-C \displaystyle \int_{\rho}^{r} g^{-1}\left( x_0,\displaystyle\frac{\text{cap}_{G(\cdot)}(\overline{B}(x_0,t) \cap \Omega^\complement; B(x_0,2t))}{t^{n-1}}\right) \, \mathrm{d}t + Cr\right).
\end{array}$$
So, by the condition $(7.1)$ for all sufficiently small $0 < \rho \leq r$, we get
$$\sup_{\Omega \cap B(x_0,\rho)}\left(H_f-f(x_0)\right) \leq 2\epsilon.$$
Then $H_f$ is continuous at $x_0$ and as $f \in C(\partial \Omega)$ was arbitrary, which implies that $x_0$ is ${G(\cdot)}$-regular.\\
For the converse, by Theorem $6.3$, we have 
$$\liminf_{x \to x_0} u(x) \leq  C \left(\displaystyle \int_{0}^{4r} g^{-1}\left( x_0,\displaystyle\frac{\text{cap}_{G(\cdot)}(\overline{B}(x_0,t) \cap \Omega^\complement; B(x_0,2t))}{t^{n-1}} \right) \, \mathrm{d}t + r\right).$$ 
By the condition $(7.1)$, we can find $r>0$ sufficiently small so that 
$$\liminf_{x \to x_0} u(x) < 1.$$
As $u$ is solution of the Sobolev-Dirichlet problem in $4B\backslash (\overline{B}\cap \Omega^\complement)$ with the continuous boundary data $1$ on $K$ and $0$ on $\partial (4B)$, then $x_0$ is not ${G(\cdot)}$-regular.
\end{proof}

\end{document}